\documentclass[graybox]{svmult}

\usepackage{mathptmx}       
\usepackage{helvet}         
\usepackage{courier}        
\usepackage{type1cm}        
\usepackage{makeidx}         
\usepackage{graphicx}        
\usepackage{multicol}        
\usepackage[bottom]{footmisc}
\usepackage{amssymb}

\makeindex             

\renewcommand{\qed}{\hfill \mbox{\raggedright \rule{0.1in}{0.1in}}}
\newcommand{\cT}{\mathcal{T}}
\newcommand{\cR}{\mathcal{R}}
\newcommand{\cL}{\mathcal{L}}
\newcommand{\cP}{\mathcal{P}}
\newcommand{\triangleL}{\triangleleft_{\mathcal{L}}}
\newcommand{\triangleLq}{\trianglelefteq_{\mathcal{L}}}
\newcommand{\triangleR}{\triangleright_{\mathcal{R}}}

\begin{document}

\title*{On the Orderability Problem and the Interval Topology}

\author{Kyriakos Papadopoulos}

\institute{Kyriakos Papadopoulos \at School of Mathematics, The University of Birmingham, Edgbaston, B15 2TT, UK \email{kxp878@bham.ac.uk}}

\maketitle


\abstract*{The class of LOTS (linearly ordered topological spaces, i.e. spaces equipped with a topology
generated by a linear order) contains many important
spaces, like the set of real numbers, the set of rational numbers and the ordinals.
Such spaces have rich topological properties, which are not necessarily hereditary.
The Orderability Problem, a very important question on whether a topological space admits
a linear order which generates a topology equal to the topology of the space, was given
a general solution by J. van Dalen and E. Wattel, in 1973. In this survey we first
investigate the role of the interval topology in van Dalen's and Wattel's characterization
of LOTS, and we then examine ways to extend this model to transitive relations that
are not necessarily linear orders.}

\abstract{The class of LOTS (linearly ordered topological spaces, i.e. spaces equipped with a topology
generated by a linear order) contains many important
spaces, like the set of real numbers, the set of rational numbers and the ordinals.
Such spaces have rich topological properties, which are not necessarily hereditary.
The Orderability Problem, a very important question on whether a topological space admits
a linear order which generates a topology equal to the topology of the space, was given
a general solution by J. van Dalen and E. Wattel, in 1973. In this survey we first
examine the role of the interval topology in van Dalen's and Wattel's characterization
of LOTS, and we then discuss ways to extend this model to transitive relations that
are not necessarily linear orders.}

\section{Introduction.}
\label{sec:1}

``{\it Order is a concept as old as the idea of number and much
of early mathematics was devoted to constructing and studying
various subsets of the real line.}'' (Steve Purisch).

In S. Purisch's account of results on orderability and suborderability (see \cite{Purisch-History}),
one can read the formulation and development of several orderability problems, starting from
the beginning of the 20th century and reaching our days. By an orderability problem, in topology, we mean the following.
Let $(X,\cT)$ be a topological space and let $X$ be equipped with an order relation $<$. Under
what conditions will $\cT_<$, i.e. the topology induced by the order $<$, be equal to $\cT$?
There was no general solution to this problem until the early 70s.

The first general solution to the characterization of LOTS (linearly ordered topological spaces, that is,
spaces whose topology is generated by a linear order)
was given by J. van Dalen and E. Wattel, in 1973 (see \cite{Dalen-Wattel}). These authors
succeeded to generalize and expand the properties that appear in the real line, where
the natural topology equals the natural order topology. Their main tool were nests.
They considered a topological space $X$, whose topology is generated by a subbase $\cL \cup \cR$ of
two nests $\cL$ and $\cR$ on $X$, whose union is $T_1$-separating. By considering also the ordering which is
generated by the nest $\cL$ on $X$, namely $\triangleL$,
the authors introduced conditions such that $\cT_{\triangleL}$ to be equal to $\cT_{\cL \cup \cR}$. In this
case, they proved the space to be LOTS, while in the case where
$\cT_{\triangleL}  \subset \cT_{\cL \cup \cR}$, the space was proved to be GO (generalized ordered, i.e.
a topological subspace of a LOTS).
Both in the case of GO-spaces and of LOTS the authors demanded
$\cL \cup \cR$ to form a $T_1$-separating subbase for the topology on $X$.
The necessary and sufficient condition for both nests $\cL$, $\cR$ to be interlocking was added in order for
the space to be LOTS.

LOTS are natural occuring topological objects and are canonical building blocks for topological
examples. For example, a space which is LOTS is also monotonically normal (see for example \cite{Lutzer-LOTS}). From the other hand,
a subspace of a LOTS is not necessarily a LOTS.

In this paper we will initially use tools from \cite{Good-Papadopoulos}, where the authors
revisited and simplified J. van Dalen and E. Wattel's ideas in order to construct ordinals, and we will investigate the role of the interval
topology in their solution to the orderability problem. The interval topology can be
defined for any transitive order (see for example \cite{compendium}), and we believe that
it is a good candidate for replacing the topology that is generated by the $T_1$-separating
union of two nests, $\cL$ and $\cR$ (the topology used by van Dalen and Wattel), in
order to extend the orderability problem to non linearly ordered spaces.

\section{Preliminaries and a Few Remarks on \cite{Dalen-Wattel}.}
\label{sec:2}

In this Section we will introduce the machinery that will be needed in order to develop our ideas
in the succeeding sections. Our main reference on standard order- and lattice-theoretic definitions
will be the book \cite{compendium}. A more recent account on topological properties of ordered
structures is given in \cite{good-bijective-preimages}.

\begin{definition}
Let $(X,<)$ be a set equipped with a transitive relation $<$. We define $\uparrow{A} \subset X$ to be the set:
\[\uparrow{A} = \{x : x \in X \textrm{ and  there exists } y \in A \textrm{, such that } y <x\}.\]
We also define $\downarrow{A} \subset X$ to be the set:
\[\downarrow{A} = \{x : x \in X \textrm{ and  there exists } y \in A \textrm{, such that } x<y\}.\]
\end{definition}

More specifically, if $A= \{y\}$, then:
\[\uparrow{A} = \{x : x \in X \textrm{ and } y<x\}\] and
\[\downarrow{A} = \{x : x \in X \textrm{ and } x<y\}\]

From now on we will use the conventions $\uparrow{x} = \uparrow{\{x\}}$ and $\downarrow{x} = \downarrow{\{x\}}$.

We remind
that the {\em upper topology} $\cT_U$ is generated
by the subbase $\mathcal{S} = \{X- \downarrow x : x \in X\}$
and the {\em lower topology} $\cT_l$ is generated by the subbase $\mathcal{S} = \{X-\uparrow x : x \in X\}$.
The {\em interval topology} $\cT_{in}$ is defined as $\cT_{in}= \cT_U \vee \cT_l$, where $\vee$ stands
for supremum.

\begin{definition}\label{definition-T0-T1}
Let $X$ be a set.
\begin{enumerate}
\item A collection $\mathcal{L}$, of subsets of $X$, $T_0$-{\em separates} $X$, if and only if for all $x,y \in X$, such that $x \neq y$, there exist $L \in
\mathcal{L}$, such that $x \in L$ and $y \notin L$ or $y \in L$
and $x \notin L$.

\item Let $X$ be a set. A collection $\mathcal{L}$, of subsets of $X$, $T_1$-{\em separates} $X$, if and only if for all $x,y \in X$, such that $x \neq y$, there exist $L,L' \in
\mathcal{L}$, such that $x \in L$ and $y \notin L$ and also $y \in
L'$ and $x \notin L'$.
\end{enumerate}
\end{definition}

One can easily see the link
between Definition \ref{definition-T0-T1} and the separation axioms of topology: a topological space $(X,\mathcal{T})$ is
$T_0$ (resp. $T_1$), if and only if there is a subbase $\mathcal{S}$, for $\mathcal{T}$, which $T_0$-separates (resp. $T_1$-separates) $X$.

\begin{definition}
Let $X$ be a set and let $\mathcal{L}$ be a family
of subsets of $X$. $\mathcal{L}$ is a {\em nest}
on $X$, if for every $M, N \in \mathcal{L}$, either $M \subset N$ or $N \subset M$.
\end{definition}

\begin{definition}\label{definition-order}
Let $X$ be a set and let $\mathcal{L}$ be a nest on $X$.
We define an order relation on $X$ via the
nest $\mathcal{L}$, as follows:
\[x \triangleleft_{\mathcal{L}} y \Leftrightarrow \,\exists \, L \in \mathcal{L}, \textrm{ such that } x \in L \textrm{ and } y \notin L\]
\end{definition}

It follows from Definitions \ref{definition-T0-T1} and \ref{definition-order} that if the nest $\cL$ is $T_0$-separating, then the ordering $\triangleL$ is linear, provided the ordering is reflexive.

We note that the declaration of reflexivity in the ordering is very vital from a
purely order-theoretic point of view: a partial ordering is defined to be reflexive,
antisymmetric and transitive. A linear ordering is a partial ordering plus every two distinct
elements in the set can be compared to one another via the ordering. On the other hand, the orderability problem is in fact a topological
problem of comparing two topologies, as stated in the introductory section, and J. van Dalen's
and E. Wattel's proof of the general solution to this problem does
not examine reflexivity or antisymmetry in the ordering: only transitivity and
comparability of any two elements. More specifically, if $x$ and $y$ are distinct
elements in a set $X$, and $\triangleL$ is an ordering relation generated by a $T_0$-separating
nest on $X$, then one can easily check that it cannot happen that $x \triangleL y$ and simultaneously
$y \triangleL x$. This gives us the liberty to say that the ordering is always antisymmetric
 but, still, from an order-theoretic point of view one should state explicitly whether the
  ordering is reflexive or not. In conclusion, the characterization of LOTS in \cite{Dalen-Wattel}
refers to linearly ordered sets, but also covers cases of sets which are equipped with an ordering $<$
which appears to be a weaker version
of the linear ordering.

Having this in mind, we find it important to make a distinction between
$\triangleL$ and $\triangleLq$.

\begin{definition}\label{definition-order-reflexive}
Let $X$ be a set and $x, y \in X$. We say that $x \triangleLq y$, if and only if
either $x=y$ or there exists $L \in \cL$, such that $x \in L$ and $y \notin L$.
\end{definition}

From now on, whenever we write $x \triangleL y$, we will assume that $x \neq y$.

The order of Definition \ref{definition-order} was first introduced
in \cite{Dalen-Wattel} and was further examined in \cite{Good-Papadopoulos},
where the authors gave the following useful (for the purposes of this paper)
theorem.

\begin{theorem}\label{theorem-kyriakos chris}
Let $X$ be a set. Suppose $\cL$ and $\cR$ are two nests on $X$. $\cL \cup \cR$ is
$T_1$-separating, if and only if $\cL$ and $\cR$ are both $T_0$-separating and
$\triangleL = \triangleR$.
\end{theorem}

\begin{definition}[van Dalen \& Wattel]\label{definiton-interlocking}
Let $X$ be a set and let $\cL \subset \cP(X)$. We say that $\cL$ is interlocking
if and only if, for each $L \in \cL$, $L = \bigcap\{N \in \cL : L \subset N,\,L \neq N\}$
implies $L = \bigcup \{N \in \cL : N \subset L,\, L \neq N\}$.
\end{definition}

The following Theorem, stemming from Definition \ref{definiton-interlocking}, gives
conditions such as a nest to be interlocking in linearly ordered spaces, in particular.

\begin{theorem}[See \cite{Good-Papadopoulos}]\label{theorem-interlocking}
Let $X$ be a set and let $\cL$ be a $T_0$-separating nest on $X$. The following are equivalent:
\begin{enumerate}

\item $\cL$ is interlocking;

\item for each $L \in \cL$, if $L$ has a $\triangleL$-maximal element, then $X-L$ has a $\triangleL$-minimal element;

\item for all $L \in \cL$, either $L$ has no $\triangleL$-maximal element or $X-L$ has a $\triangleL$-minimal
element.

\end{enumerate}
\end{theorem}

Theorem \ref{theorem-interlocking} permits us to say that if $\cL$ is a $T_0$-separating nest on $X$ and if for every $L \in \cL$, $X-L$ has a minimal element,
then $\cL$ is interlocking.

We remark that the subset $X = [0,1) \cup \{2\}$, of the real line, together with its subspace topology inherited from the topology
of the real line, is a non-compact GO-space, but not a LOTS: the ordering $\triangleL$, where $\cL = X \cap \{(-\infty,a): a \in \mathbb{R}\}$, cannot
``spot'' the difference between $X$ and the space $Y = [0,2]$, because it cannot
tell whether there is a gap between $[0,1)$ and $2$ or not. The property of interlocking (Theorem \ref{theorem-interlocking}) is
what guarantees that there are not such gaps. We will now give the version of the solution to the
orderability problem that was stated by van Dalen and Wattel, as it appeared in \cite{Good-Papadopoulos}.

\begin{theorem}[van Dalen \& Wattel]\label{theorem-LOTS-GO}
Let $(X,\mathcal{T})$ be a topological space. 
\begin{enumerate}

\item If $\mathcal{L}$ and $\mathcal{R}$ are two nests of open
sets, whose union is $T_1$-separating, then every $\triangleleft_{\mathcal{L}}$-order
open set is open, in $X$.

\item $X$ is a GO space, if and only if there are two nests, $\mathcal{L}$ and $\mathcal{R}$,
of open sets, whose union is $T_1$-separating and forms a subbase for $\cT$.

\item $X$ is a LOTS, if and only if there are two interlocking nests $\cL$ and $\cR$,
of open sets, whose union is $T_1$-separating and forms a subbase for $\cT$.

\end{enumerate}
\end{theorem}

In the section that follows, we will examine particular properties of the
interval topology, when the ordering is generated by a $T_0$-separating nest,
and we will see the key role that it plays in the characterization of LOTS.
The topology $\cT_{\cL \cup \cR}$ does not have any particular topological
meaning when the union of $\cL$ and $\cR$ is not $T_1$-separating. The interval topology
though, as being more flexible from the way it is defined, can replace the $\cT_{\cL \cup \cR}$ topology, if subjected to
certain conditions. We will see this in more detail in Section 3 that follows.

\section{Some Further Remarks on the Orderability Problem.}
\label{sec:3}

\begin{remark}\label{remark-interval-in-T0-2-cases}
Let $X$ be a set and let $\cL$ be a $T_0$-separating nest on $X$.

If $\triangleLq$ is reflexive, then obviously the interval topology $\cT_{in}^{\triangleLq}$, via $\triangleLq$, will be
equal to the topology $\cT_{\triangleL}$ (for a rigorous proof, one should add $T_0$-separation
in Lemma \ref{lemma-lower-upper-via-nests} of Section 4). In addition, the order topology $\cT_{\triangleLq}$
will be equal to the discrete topology on $X$.

If $\triangleL$ is irreflexive, then the interval
topology via $\triangleL$ will be equal to the discrete topology on $X$. Indeed, $\downarrow{a} = \{x \in X : x \triangleL a\}$
and so $X- \downarrow{a} = \{x \in X : a \triangleLq x\} = (-\infty,a]$. In
a similar fashion, $x- \uparrow{a} = [a,\infty)$ and so $(-\infty,a] \cap [a,\infty) = \{a\}$.
\end{remark}

Given the conditions in Theorem \ref{theorem-LOTS-GO} and the observations in Remark \ref{remark-interval-in-T0-2-cases}, we get the following comparisons for the topologies $\cT_{\triangleL},\cT_{\triangleLq},\cT_{in}^{\triangleL},\cT_{in}^{\triangleLq}$, on a space $X$:

\begin{lemma}\label{lemma-comparing-all-topologies-vanDalen}
\begin{enumerate}

\item $\cT_{\triangleL} = \cT_{\cL \cup \cR} = \cT_{in}^{\triangleLq} \subseteq \cT_{in}^{\triangleL} = \cT_{\triangleLq}$,
provided that $\cL$ and $\cR$ are interlocking and $\cL \cup \cR$ $T_1$-separates $X$.

\item $\cT_{\triangleL} = \cT_{in}^{\triangleLq} \subseteq \cT_{\cL \cup \cR} \subseteq \cT_{in}^{\triangleL} = \cT_{\triangleLq}$,
provided that $\cL \cup \cR$ $T_1$-separates $X$.

\end{enumerate}
\end{lemma}

Lemma 1 permits us to restate Theorem \ref{theorem-LOTS-GO}, using the interval topology.
\begin{corollary}
A topological space $(X,\cT)$ is:
\begin{enumerate}

\item a LOTS, iff there exists a nest $\cL$ on $X$, such that $\cL$ is $T_0$-separating and interlocking
and also $\cT = \cT_{in}^{\triangleLq}$.

\item a GO-space, iff there exists a nest $\cL$ on $X$, such that $\cL$ is $T_0$-separating
and also $\cT = \cT_{in}^{\triangleLq}$.

\end{enumerate}
\end{corollary}

Our proposed weaker version of the orderability problem, that will be presented Section 5,
 will be based on observations on Lemma \ref{lemma-comparing-all-topologies-vanDalen}. Knowing that the interval topology
 can be defined via any transitive relation, a weaker version of the orderability question can be expressed
 as follows.

{\bf Question:} Let $X$ be a set equipped with a transitive relation $<$ and the interval topology $\cT_{in}^\le$,
 defined via $\le$, where $\le$ is $<$ plus reflexivity. Under which necessary and sufficient conditions will $\cT_{<}$ be equal to $\cT_{in}^{\le}$?

 In this paper we give a partial answer to this question, through Theorem \ref{theorem-stronger-orderability}.

\section{The Order and the Interval Topologies in the Light of Nests.}
\label{sec:4}

In order to give an answer to the Question of Section 3, we will need first to see what form do
the topologies $\cT_{\triangleL}$ and $\cT_{in}^{\triangleLq}$ take, when the nest $\cL$ is not
necessarily $T_0$-separating.

\begin{lemma}\label{lemma-orders-via-nests}
Let $X$ be a set and let $\cL$ be a nest on $X$. Let also $\Delta = \{(x,x): x \in X\}$. Then:
\begin{enumerate}
\item $\triangleL = \bigcup_{L \in \cL} [L \times (X-L)]$.

\item $\ntriangleleft_{\cL} = X \times X - \triangleL = \bigcap_{L \in \cL} [((X-L) \times X)  \cup (X \times L)]$.

\item $\ntrianglelefteq_{\cL} = \bigcap_{L \in \cL}[((X-L)\times X) \cup (X \times L)] \cap (X-\Delta)$

\end{enumerate}
\end{lemma}

{\bf Notation:} From now on, if $U \subset X \times X$, then $U(x) = \{y \in X : (x,y) \in U\}$.

\begin{lemma}\label{lemma-lower-upper-via-nests}
\begin{enumerate}

\item For each $x \in X$, $X - \uparrow x = \bigcap\{L \in \cL : x \in L\} - \{x\}$.

\item For each $x \in X$, $X- \downarrow x = \bigcap\{X-L : x \in X-L\} \cup \{x\}$
\end{enumerate}
\end{lemma}
\begin{proof}
\begin{enumerate}
\item $y \in X - \uparrow x$, if and only if $x \ntrianglelefteq_{\cL} y$, if and only
if $(x,y) \notin \triangleLq$, if and only if (by Lemma \ref{lemma-orders-via-nests})
$y \in \bigcap_{L \in \cL} \{((X-L) \times X) (x)  \cup (X \times L) (x)\}$ and $y \neq x$,
if and only if $y \in \bigcap \{L \in \cL : x \in X\}-\{x\}$.

\item In a similar fashion to 1. using Lemma \ref{lemma-orders-via-nests}.
\end{enumerate}
\qed
\end{proof}

Let us now have a closer look to the order topology. It is known
that if a set $X$ is equipped with an order $<$, then the order topology $\cT_<$, on $X$,
will be the supremum of two topologies, namely the topology $\cT_\leftarrow$ and the
topology $\cT_\rightarrow$. In particular, $\cT_\leftarrow$ is generated by a subbasis
$\mathcal{S}_\leftarrow = \{(\leftarrow,a): a \in X\}$, where $(\leftarrow,a)= \{x \in X : x <a\}$.
Similarly, the topology $\cT_\rightarrow$ is generated by a subbasis $\mathcal{S}_\rightarrow = \{(a,\rightarrow): a \in X\}$,
where $(a,\rightarrow) = \{x \in X : a <x\}$. We will now see how these subbases look
like when the ordering is generated by a not necessarily $T_0$-separating nest, that is $< = \triangleL$,
where $<$ stands for any (not necessarily linear) order.

{\bf Notation:} Let $\cL$ be a nest on a set $X$. Then, $\cL_a = \{L \in \cL : a \in L\}$, for each $a \in X$.

\begin{lemma}\label{lemma-order-generalization}
\begin{enumerate}
\item For each $a \in X$, $(a,\rightarrow) = \bigcup_{L \in \cL_a} X-L$\,.

\item For each $a \in X$, $(\leftarrow,a) = \bigcup_{L \notin \cL_a} L$\,.

\end{enumerate}
\end{lemma}

It is easily seen that when $\cL$ is $T_0$-separating, Lemmas \ref{lemma-orders-via-nests} and \ref{lemma-lower-upper-via-nests}
compensate to our remarks in Section 3.

\section{A Weaker Orderability Problem.}
\label{sec:5}

In this Section we find necessary conditions, so that $\cT_{in}^\le = \cT_<$, where
$<$ is a transitive relation, $\le$ is $<$ plus reflexivity and $\cT_{in}^\le$ is
the interval topology that is defined via $\le$. To achieve this, we will give
conditions such that $\cT_l = \cT_\leftarrow$ and $\cT_U = \cT_\rightarrow$.
The logic behind these conditions is the following. As we have seen in Theorem \ref{theorem-kyriakos chris},
linearity in the space is strongly related to $T_0$-separating nest. We believe that a first step towards the generalization of the orderability
problem will be to define a weaker version of $T_0$-separation, and we can achieve
this without using the notion of nest. It seems that the use of nests was vital
in the understanding and the description of the order-theoretic and of the many topological
properties of linearly ordered topological spaces. Even the fact that our approach
can be stated using nests and the material that was presented in Section 4,
we prefer to adopt a more general approach. The description of the order topology and
the interval topology in Section 4 will permit us to see that (i) if the nest
is $T_0$-separating we will get back to Section 3
and that (ii) nests might not help that much to describe more complex
ordering relations. This can be also seen in the topology $\cT_{\cL \cup \cR}$, which 
loses its meaning when it does not refer to linearly ordered sets.
\\

Let $X$ be a set and let $<$ be a transitive relation on $X$.

{\bf Condition 1} Let $x,y \in X$, such that $x \nleq y$. Then, there exist $z_i \in X$, $i=1,\cdots,n$,
such that $y < z_i$ and, if $w \in X$, such that $w < z_i$, then $x \nleq w$.
\\

{\bf Condition 2} Let $x,y \in X$, such that $y \nleq x$. Then, there exist $z_i \in X$, $i=1,\cdots,n$,
such that $z_i <y$ and, if $w \in X$, such that $z_i <w$, then $w \nleq x$.
\\

{\bf Condition 3} Let $x,y \in X$, such that $y <x$. Then, there exist $z_i \in X$, $i=1,\cdots,n$,
such that $z_i \nleq y$ and, if $w \in X$, such that $z_i \nleq w$, then $w<x$.
\\

{\bf Condition 4} Let $x,y \in X$, such that $x<y$. Then, there exist $z_i \in X$, $i=1,\cdots,n$,
such that $y \nleq z_i$ and, if $w \in X$, such that $w \nleq z_i$, then $x <w$.

\begin{proposition}\label{proposition-segment-equalities}
\begin{enumerate}
\item $\cT_l = \cT_\leftarrow$, if and only if Conditions 1 and 3 are satisfied.

\item $\cT_U = \cT_\rightarrow$, if and only if Conditions 2 and 4 are satisfied.
\end{enumerate}
\end{proposition}
\begin{proof}
$\cT_l \subset \cT_\leftarrow$, if and only if for an arbitrary $y \in X - \uparrow x$, there
exists a basic-open set $B$, in $\cT_\leftarrow$, such that $y \in B \subset X - \uparrow x$. By
Condition 1, there exist $z_i \in X$, $i=1,\cdots,n$, such that $y \in \bigcap_{i=1}^n (\leftarrow,z_i) \subset
X- \uparrow x$. That $\cT_\leftarrow \subset \cT_l$ follows in a similar fashion, using Condition 3
and given this, the proof of $\cT_U = \cT_\rightarrow$ is straightforward.
\qed
\end{proof}

\begin{theorem}\label{theorem-stronger-orderability}
If Conditions 1,2,3 and 4 are satisfied, then $\cT_{in}^\le = \cT_<$.
\end{theorem}

\begin{remark}\label{remark-2}
If $(X,<)$ is a linearly ordered set, then Condition 1 is
obviously satisfied. Indeed, if $x \nleq y$, then $y<x$. So, there exists $z=x$,
such that $y<z$ and, if $w \in X$, such that $w<x$, then $x \nleq w$.
\end{remark}

Remark 2 shows that our conditions give weaker properties than those satisfied
from linear orders (and $T_0$-separating nests). Even the fact that we have a necessary but not
yet a sufficient condition for our Question of Section 3, Theorem \ref{theorem-stronger-orderability}
permits us to conclude that the ``distance'' of the interval topology $\cT_{in}^{\triangleLq}$, from the order topology $\cT_<$,
depends on how weaker are our conditions from $T_0$-separation (-linear order).


\end{document}